\documentclass[12pt,a4]{amsart}

\usepackage{color}
\usepackage{hyperref}

\newtheorem{thm}{Theorem}
\newtheorem{lem}[thm]{Lemma}
\newtheorem{crl}[thm]{Corollary}
\newtheorem{rmk}{Remark}

\numberwithin{equation}{section}
\numberwithin{thm}{section}
\numberwithin{rmk}{section}
\numberwithin{nt}{section}
\numberwithin{Example}{section}

\makeatletter
\let\oldcite\cite
\def\cite#1{\expandafter\def\csname#1\endcsname{}\oldcite{#1}}
\let\oldbibitem\bibitem
\def\bibitem#1{\@ifundefined{#1}{\typeout{#1 nonused}}
{\typeout{#1 used}}\oldbibitem{#1}}
\makeatother

\newcommand{\B}[1]{{\mathbb{#1}}}
\newcommand{\f}[1]{{\boldsymbol{#1}}}

\newcommand{\wha}[1]{\widehat{#1}}

\font\rus=wncyr10 %scaled 1200
\DeclareMathOperator{\Ker}{{{Ker}}}
\DeclareMathOperator{\im}{{{Im}}}

\DeclareMathOperator{\LGen}{{\mathtt{LGen}}}

\newcommand{\alp}{\alpha}
\newcommand{\gam}{\gamma}

\newcommand{\lam}{\lambda}
\newcommand{\Lam}{\Lambda}
\newcommand{\ome}{\omega}
\newcommand{\Ome}{\Omega}

\newcommand{\der}{\partial}
\newcommand{\ten}{\otimes}

\newcommand{\wed}{\wedge}

\newcommand{\con}{\,\lrcorner\,}

\newcommand{\Fla}{^{\flat}{}}
\newcommand{\Sha}{^{\sharp}{}}
\newcommand{\END}{{\,\text{\footnotesize$\square$}}}
\newcommand{\db}[1]{\,{{#1}\!{#1}}\,}

\newcommand{\M}[1]{{\mathcal{#1}}}
\newcommand{\K}[1]{\text{\rus{#1}}}

\newcommand{\com}{{}\circ{}}
\newcommand{\Ele}{^\mathfrak{e}{}}
\newcommand{\Elg}{^\mathfrak{g}{}}
\newcommand{\Joi}{^\mathfrak{j}{}}
\newcommand{\tecka}{\text{\Large{.}}}

\begin{document}
%--------------------------------------------------------------------%
\markboth{J. Jany\v{s}ka}
{Local Lie algebras of pairs}

\title[Remarks on Local Lie algebras of pairs]{Remarks on Local Lie algebras 
\\
of pairs of functions}

\author{Josef Jany\v ska}

\thanks{Supported
by the  grant GA \v CR 14--02476S}

\address{\ 
\newline
Department of Mathematics and Statistics, Masaryk University
\newline
Kotl\'a\v rsk\'a 2, 611 37 Brno, Czech Republic
\newline
e-mail: janyska@math.muni.cz}
  
 \begin{abstract}
We study local Lie algebras of 
pairs of functions which generate infinitesimal symmetries of almost-cosymplectic-contact structures of  odd dimensional manifolds.
\end{abstract} 
  
%--------------------------------------------------------------------%
\maketitle
%--------------------------------------------------------------------%

\bigskip

\noindent{{\it Keywords}:
Almost-cosymplectic-contact structure; almost-coPoisson-Jacobi structure; infinitesimal symmetry; Lie algebra.
\\
{\ }
\\
{\it Mathematics Subject Classification 2010}:
53C15, 53B99, 17B66.
%70H40, % Relativistic Mechanics
%70H45, % Constrained Mechanics
%70H33, % Symmetries and Conservation Laws
%70G45, % Diff. Geom. methods in mechanics
%58A20. % Glob. analysis - Jets
}

%--------------------------------------------------------------------%
\section{Introduction}
\setcounter{equation}{0}
%--------------------------------------------------------------------%
It is very well known, \cite{Kir76}, that (1-dimensional) local Lie algebras of functions on a $(2n+1)$-dimensional manifold $ \f M $ are in one-to-one correspondence with  \emph{Jacobi structures} on $ \f M $ given by a skew symmetric 2-vector field $ \Lambda $ and a vector field $ E $ such that 
\begin{equation*}%\label{Eq: 1.1}
[E,\Lam] = 0\,, \qquad [\Lambda,\Lambda] = - 2\, E\wedge\Lambda\,,
\end{equation*}
where $[,]$ is the Schouten-Nijenhuis bracket (see, for example, \cite{Vai94}) of skew symmetric multi-vector fields.

 Then the \emph{Jacobi bracket}
\begin{equation*}%\label{Eq: 1.2}
[f,h] = \{f,h\} - f\, E\tecka h + h \, E\tecka f\,,
\end{equation*}
where $\{f,h\} = \Lambda(df,dh)$ is the \emph{Poisson bracket}, defines on the sheaf of functions $C^\infty(\f M)$ the structure of  a local (Jacobi) Lie algebra.

Let us assume the subsheaf $C_E^\infty(\f M)$ of functions constant on integral curves of the vector field $ E $, i.e. such that $E\tecka f =0$ (conserved functions in the terminology of \cite{JanVit12}).
 Then the restriction of the above Jacobi bracket defines the Lie subalgebra $(C_E^\infty(\f M); \{,\}) \subset (C^\infty(\f M); [,]) $ of conserved functions. Really, if $ f,h \in C_E^\infty(\f M) $ then
$$  
E \tecka \{ f,h\} = \{E \tecka f , h\} + \{ f , E\tecka h\} = 0\,.
$$
Moreover, the \emph{Hamilton-Jacobi lift}
\begin{equation*}%\label{Eq: 1.3}
X_f = df\Sha - f\, E\,
\end{equation*}
of a function $f\in C_E^\infty(\f M)$ is an infinitesimal symmetry of $(E,\Lambda)$, i.e. $L_{X_f} E = [X_f,E] = 0$ and
$L_{X_f}\Lambda = [X_f,\Lambda] =0$.
In what follows we shall use the notation $\Lambda\Sha(\alpha) := \alpha\Sha = i_{\alpha}\Lambda$ for any 1-form $\alpha$.

If the pair $(E,\Lam)$ is \emph{regular} (\emph{transitive} in the terminology of \cite{Kir76}), i.e. $E\wedge \Lam^n \not\equiv 0$,
then there exist the unique 1-form $\omega$ such that $i_E \omega = 1$ and $i_\omega\Lambda= 0$. Moreover, $\omega$ is a \emph{contact form}, i.e. $\omega\wedge d\omega^n \not\equiv 0$ and $E$ is the Reeb vector field of the contact structure  $(\omega,\Omega= d\omega)$, i.e. $i_E\Omega = 0$. 
The pairs  $(\omega,\Omega)$ and $(E,\Lam)$ are said to be \emph{mutually dual}. 
It is easy to see that the  Hamilton-Jacobi lift of a conserved function is an infinitesimal symmetry of the contact pair
 $(\omega,\Omega)$, i.e. $L_{X_f} \omega = 0$ and $L_{X_f}\Omega = 0$.
 The Hamilton-Jacobi lift of conserved functions is a Lie algebra homomorphism from the Lie algebra $(C^\infty_E(\f M), \{,\})$
 to the Lie algebra $(\M L(\omega,\Omega), [,]) \subset (\M X(\f M),[,])$ of infinitesimal symmetries of the contact structure $(\omega,\Omega)$ (respective to the Lie algebra $(\M L(E,\Lambda),[,])$ of infinitesimal symmetries of the Jacobi structure $(E,\Lambda)$).
Here $ (\M X(\f M),[,]) $ is the Lie algebra of vector fields.

\smallskip
In \cite{JanMod09} dual contact and Jacobi structures were generalized in the following sense.
An 
\emph{almost-cosymplectic-contact (regular) structure (pair)}
is
given by
a pair 
$(\ome,\Ome)$
such that
\begin{equation*}%\label{Eq: 1.4}
d\Ome = 0 \,,\qquad
\ome \wed \Ome^n \not\equiv 0 \,.
\end{equation*}
According to \cite{Lic78} there exists a unique dual 
\emph{almost-coPoisson-Jacobi structure (pair)\/} 
given by a
pair $(E,\Lam)$ such that
\begin{equation*}%\label{Eq: 1.5}
(\Ome\Fla_{|\im(\Lam\Sha)})^{-1}
= \Lam\Sha_{|\im(\Ome\Fla)} \,,
\quad i_E\ome =1\,,
\quad i_E\Ome =0\,,
\quad i_\ome\Lam = 0\,,
\end{equation*}
where $\Omega\Fla:T\f M \to T^*\f M$ is given by $ \Omega\Fla(X) := X\Fla = i_X\Omega$. 
Then (see \cite{JanMod09})
\begin{equation*}%\label{Eq: 1.6}
[E,\Lam] = - E\wed \Lam\Sha(L_E\ome)\,,
\qquad
[\Lam,\Lam] = 2 \, E\wed (\Lam\Sha\ten\Lam\Sha)(d\ome)\,.
\end{equation*}

\begin{rmk}\label{Rm: 1.1}
{\rm
The almost-cosymplectic-contact pair and the dual al\-most-coPoisson-Jacobi pair generalize not only the contact and the dual Jacobi structures but also cosymplectic structures.
Really, if $d\ome =0$ we obtain a \emph{cosymplectic pair} (see, for example, \cite{deLTuy96}).
The corresponding dual pair is \emph{coPoisson pair}, \cite{JanMod09},
given by the pair
$(E, \Lam)$
such that
$
[E, \Lam] = 0 \,,
$
$
[\Lam, \Lam] = 0 \,.
$
}
\hfill\END
\end{rmk}

\begin{rmk}\label{Rm: 1.2}
{\rm
For cosymplectic and contact structures we have distinguished forms $ \omega,\,\, \Omega$ and vector fields  $E,\,\, \Lambda $. But for an almost-cosymplectic-contact structure we have distinguished forms $ \omega, L_E\omega,$ $\Omega,$ $ d\omega $ and distinguished vector fields $ E, (L_E\omega)\Sha, \Lambda, (\Lambda\Sha\otimes\Lambda\Sha)(d\omega) $.
}
\hfill\END
\end{rmk}

In what follows we assume an odd dimensional manifold $ \f M $ with a regular almost-cosymplectic-contact structure $ (\omega,\Omega)  $. We assume the dual (regular) almost-coPoisson-Jacobi structure  $(E,\Lambda)$. 
 Then we have $ \Ker (\omega) = \im(\Lambda\Sha) $ and $ \Ker (E) = \im(\Omega\Fla) $ and we have the splitting
$$
T\f M =  \im(\Lambda\Sha)\oplus \langle E \rangle\,, \qquad
T^*\f M = \im(\Omega\Fla) \oplus \langle \omega \rangle\,,
$$ 
i.e. any vector field $X$ and any 1-form $\beta$ can be decomposed as
\begin{equation}\label{Eq: 1.7}
X = X_{(\alpha,h)} = \alpha\Sha + h\, E\,, \qquad
\beta = \beta_{(Y,f)} = Y\Fla + f\, \omega\,,
\end{equation}
where $h,f\in C^\infty(\f M)$, $\alpha$ be a 1-form and $Y$ be a vector field. 
Moreover, $h = \omega(X_{(\alpha,h)})$ and $f = \beta_{(Y,f)}(E)$. Let us note that the splitting \eqref{Eq: 1.7} is not defined uniquely, really $ X_{(\alpha_1,h_1)} = X_{(\alpha_2,h_2)}$ if and only if $\alpha_1\Sha = \alpha_2\Sha$ and $h_1= h_2$, i.e. $\alpha_1\Sha - \alpha_2\Sha = 0$ that means that $\alp_1 - \alpha_2 \in \langle\omega\rangle$. Similarly
$\beta_{(Y_1,f_1)} = \beta_{(Y_2,f_2)}$ if and only if $Y_1 - Y_2 \in \langle E \rangle$ and $f_1 = f_2$.

The projections $p_2:T\f M \to \langle E \rangle$ and $p_1: T\f M \to \im(\Lambda\Sha) = \Ker(\omega)$ are given by $ X \mapsto \omega(X) \, E $ and $ X \mapsto X - \omega(X) \, E $\,. Equivalently, the projections $q_2:T^{*}\f M \to \langle \omega \rangle$ and $q_1: T^{*}\f M \to \im(\Omega\Fla) = \Ker(E)$ are given by $ \beta \mapsto \beta(E) \, \omega $ and $ \beta \mapsto \beta - \beta(E) \, \omega $\,.
Moreover, $\Lambda\Sha \com \Omega\Fla = p_1$ and $\Omega\Fla \com \Lambda\Sha = q_1$.

\begin{rmk}\label{Rm: 1.3}
{\rm
Let us consider the pair $(\omega,F)$ where $F= \Omega + d\omega$ is a closed 2--form. Then the pair $(\omega,F)$ is not generally an almost-cosymplectic-contact pair because it may not to be regular. We shall use the form $F$ later in Theorem \ref{Th: 3.1}.
}
\hfill\END
\end{rmk}

In \cite{Jan16} we have studied infinitesimal symmetries of tensor fields $\omega$, $\Omega$, $E$, $\Lambda$ generating the almost-cosymplectic-contact and the dual almost-coPoisson-Jacobi structures.
Such symmetries are vector fields of the type $X_{(\alpha,h)} = \alpha\Sha + h \, E$, where $\alpha$ and $h $ meet certain conditions and the fact that they generate infinitesimal symmetries of a tensor field defines a Lie algebra structure on some subsheaf of $\Omega^1(\f M) \times C^\infty(\f M)$ of generators $ (\alpha,h) $ of infinitesimal symmetries.

In this paper we shall study the situation where the 1-form $\alpha$ is locally exact, i.e. $\alpha = df$ for $f\in C^\infty(\f M)$. This leads to local Lie algebras of pairs of functions generating infinitesimal symmetries of some tensor fields. Such Lie algebras are 2-dimensional local Lie algebras in the sense of \cite{Kir76}.  

\smallskip
Generally, we can define local Lie algebra structure in $C^{\infty}(\f M) \times C^{\infty}(\f M)$ by a (double) bracket $\db[(f_1,h_1);(f_2,h_2)\db]$
in $C^{\infty}(\f M)\times C^{\infty}(\f M)$ satisfying the following conditions:

(1) it defines a Lie algebra structure in $C^{\infty}(\f M)\times C^{\infty}(\f M)$ over $\B R$,

(2) $\db[(f_1,h_1);(f_2,h_2)\db]$ is continuous in $f_i$ and $h_i$, $i= 1,2$,

(3) $\mathrm{supp}\db[(f_1,h_1);(f_2,h_2)\db]\subset \mathrm{supp}(f_1,h_1)\cap \mathrm{supp}(f_2,h_2)$ for each $(f_1,h_1)$ and $(f_2,h_2)$, where $\mathrm{supp}(f_i,h_i)=\mathrm{supp}(f_i)\cap \mathrm{supp}(h_i)$.

\smallskip
As an example we can consider two Jacobi pairs $(E_i,\Lambda_i)$, $i=1,2$,
and the corresponding Jacobi brackets $[,]_i$. Then
the bracket 
$$
\db[(f_1,h_1);(f_2,h_2)\db] = ([f_1,f_2]_1;[h_1,h_2]_2)
$$ 
defines the Lie algebra structure in $C^{\infty}(\f M)\times C^{\infty}(\f M)$. As far as we know, the classification of local Lie algebras of pairs of functions is not known. In the paper we shall describe several subsheafs of  $C^{\infty}(\f M)\times C^{\infty}(\f M)$ with a Lie algebra structure given by the fact that pairs of functions generate infinitesimal symmetries of basic tensor fields
$ \omega,\, \Omega,\, E,\, \Lambda $ given by the almost-cosymplectic-contact and the dual almost-coPoisson-Jacobi structures.

\smallskip
All manifolds and mappings are assumed to be smooth.

%--------------------------------------------------------------------%
\section{Local Lie algebras of generators of infinitesimal symmetries}
\setcounter{equation}{0}
%--------------------------------------------------------------------%

In this section we shall study local Lie algebras of pairs of functions which generate infinitesimal symmetries of basic tensor fields.

 For two functions $ f,h\in C^\infty(\f M) $ we define their {\it pre-Hamiltonian lift} to vector fields on $ \f M $ by 
\begin{equation}\label{Eq: 2.1}
X_{(f,h)} = df\Sha + h\, E\,.
\end{equation} 
 
\begin{lem}\label{Lm: 2.1}
Let $(f_i,h_i) \in C^\infty(\f M) \times  C^\infty(\f M)$, $i= 1,2,$ be two pairs of functions on $\f M$. Then
\begin{align}\label{Eq: 2.2}
& [X_{(f_1,h_1)},  X_{(f_2,h_2)}]
 =
\big(
d\{f_1,f_2\} + (E{\tecka}f_1)\,(L_{df_2\Sha} + h_2\, L_E)\omega
\\
& \quad\nonumber
- (E{\tecka}f_2)\,(L_{df_1\Sha} + h_1\, L_E)\omega
- h_2\, d(E{\tecka}f_1) + h_1\, d(E{\tecka}f_2)
\big)\Sha
\\
& \quad\nonumber
+ \big(
 \{f_1,h_2\} - \{f_2,h_1\} - d\omega(df_1\Sha,df_2\Sha) 
\\
& \quad\nonumber
 + h_1\, (E{\tecka}h_2 + \Lambda(L_E\omega,df_2)) 
 - h_2\, (E{\tecka}h_1 + \Lambda(L_E\omega,df_1 ))
\big)\, E\,.
\end{align}
 \end{lem}

\begin{proof}
We have (see \cite{JanMod09})
\begin{align}
	[E,df\Sha] \label{Eq: 2.3}
&=
\big( d(E{\tecka}f) - (E{\tecka}f) \, (L_E\ome)\big)\Sha
	+ \Lam(L_E\ome,df) \, E \,,
\\
	[df\Sha, dh\Sha] \label{Eq: 2.4}
&=
	\big(d\Lam(df,dh)
+  (E{\tecka}f) \, (i_{dh\Sha}d\ome)
\\ 
& \quad\nonumber
-  (E{\tecka}h) \, (i_{df\Sha}d\ome)\big)\Sha 
-  d\ome(df\Sha,dh\Sha) \, E \,
\end{align}
which implies \eqref{Eq: 2.2}.
\end{proof}

It is easy to see that the vector field \eqref{Eq: 2.2} is not generally the pre-Hamiltonian lift of a pair of functions. It is so in the case when the projection $p_1: T\f M \to \Ker(\omega)$ of \eqref{Eq: 2.2} is the $\Lambda\Sha$-lift of the differential of a function. We shall describe several examples of subsheafs of $C^\infty(\f M) \times  C^\infty(\f M)$ 
such that the Lie bracket of two pre-Hamiltonian lifts of pairs of functions from the subsheaf is the pre-Hamiltonian lift of a pair from the subsheaf.  All such subsheafs are given by local generators of infinitesimal symmetries of basic tensor fields.

%--------------------------------------------------------------------%
\subsection{Infinitesimal symmetries of $ \omega $ and $\Omega$ generated by pairs of functions}
\label{Sec: omega pais of functions}
%\setcounter{equation}{0}
%--------------------------------------------------------------------%

\begin{thm}\label{Th: 2.2}
The pre-Hamiltonian lift \eqref{Eq: 2.1} is an infinitesimal symmetry of $ \omega $ if and only if 
\begin{equation}\label{Eq: 2.5}
i_{df\Sha}d\omega + h\, i_E\, d\omega + dh = 0\,.
\end{equation}
\end{thm}

\begin{proof}
From $ i_E \omega = 1 $ and $ i_{df\Sha} \omega = 0 $ we get
$$  
L_{X_{(f,h)}} \omega = i_{df\Sha}d\omega + h\, i_E\, d\omega + dh
$$
which proves Theorem \ref{Th: 2.2}.
\end{proof}

Let us denote by $\LGen(\omega)\subset C^\infty(\f M) \times  C^\infty(\f M)$ the subsheaf of pairs of functions $(f,h)$
on $\f M$ which generate (locally) infinitesimal symmetries of $\omega$, i.e. satisfy the condition \eqref{Eq: 2.5}.

\begin{thm}\label{Th: 2.3}
The Lie bracket of the pre-Hamiltonian lifts of two pairs $(f_i,h_i) \in \LGen(\omega)$, $i= 1,2,$ is the pre-Hamiltonian lift of a pair of functions from $\LGen(\omega)$.
\end{thm}

\begin{proof}
If \eqref{Eq: 2.5} is satisfied then $ dh_i = - (i_{df_i\Sha} + h_i\,i_E)\, d\omega = - (L_{df_i\Sha} + h_i \, L_E)\omega $ and  $E{\tecka}h_i + \Lambda(L_E\omega,df_i) = 0$ which follows by evaluating \eqref{Eq: 2.5} on $E$. Then we can rewrite \eqref{Eq: 2.2} as
\begin{align}\label{Eq: 2.6}
 [X_{(f_1,h_1)}, &  X_{(f_2,h_2)}]
  = 
\big(
d(\{f_1,f_2\} - h_2\,(E{\tecka}f_1)
+ h_1\, (E{\tecka}f_2))
\big)\Sha
\\
& \quad\nonumber
+ \big(
 \{f_1,h_2\} - \{f_2,h_1\} - d\omega(df_1\Sha,df_2\Sha) 
\big)\, E
\,
\end{align}
which is the pre-Hamiltonian lift of the pair
\begin{align}\label{Eq: 2.7}
\big(
\{f_1,f_2\} - h_2\,(E{\tecka}f_1)
+
&
 h_1\, (E{\tecka}f_2);
\\
&
 \{f_1,h_2\} - \{f_2,h_1\} - d\omega(df_1\Sha,df_2\Sha) 
\big)\,. \nonumber
\end{align}
The pair \eqref{Eq: 2.7} is in $\LGen(\omega)$ which follows from the fact that, according to Theorem \ref{Th: 2.2}, the pre-Hamiltonian lifts \eqref{Eq: 2.1} of pairs from $\LGen(\omega)$ are infinitesimal symmetries of $\omega$. Then from $L_{[X,Y]} = L_X L_Y - L_YL_X$ the Lie bracket \eqref{Eq: 2.6} of two pre-Hamiltonian lifts of pairs from $\LGen(\omega)$ is an infinitesimal symmetry of $\omega$ and, by Theorem \ref{Th: 2.2}, the pair \eqref{Eq: 2.7} has to satisfy the condition \eqref{Eq: 2.5}, i.e. it is in $\LGen(\omega)$.
\end{proof}

As a consequence we obtain the Lie bracket
\begin{align}\label{Eq: 2.8}
\db[(f_{1},h_{1}) 
;(f_{2},h_{2})\db]
 = &
\big(
\{f_1,f_2\} - h_2\,(E{\tecka}f_1)
+
 h_1\, (E{\tecka}f_2);
\\
&
 \{f_1,h_2\} - \{f_2,h_1\} - d\omega(df_1\Sha,df_2\Sha) 
\big)\,\nonumber 
 \end{align}
 which defines the local Lie algebra structure on $\LGen(\omega)$.
Moreover, the pre-Hamiltonian lift \eqref{Eq: 2.1} is the Lie algebra homomorphism from the local Lie algebra $ (\LGen(\omega); \db[,\db]) $ to the Lie algebra $(\mathcal{L}(\omega); [,])\subset (\mathcal{X}(\f M); [,])$ of infinitesimal symmetries of $\omega$.

\smallskip

Now, we shall define a Lie algebra structure on the subsheaf 
$ C^\infty_E(\f M) \times  C^\infty(\f M) $
of pairs $(f,h)$ of functions, where $f$ is conserved.

\begin{thm}\label{Th: 2.4}
The Lie bracket of the pre-Hamiltonian lifts of two pairs $(f_i,h_i) \in C^\infty_E(\f M) \times  C^\infty(\f M)$, $i= 1,2,$ is the pre-Hamiltonian lift of a pair of functions from $C^\infty_E(\f M) \times  C^\infty(\f M)$.
\end{thm}

\begin{proof}
For $E{\tecka}f_i = 0$ the 
vector field \eqref{Eq: 2.2} has the expression
\begin{align*}
[X_{(f_1,h_1)} &, X_{(f_2,h_2)}]
 =
\big(
d\{f_1,f_2\} 
\big)\Sha
\\
& \nonumber
+ \big(
 \{f_1,h_2\} - \{f_2,h_1\} - d\omega(df_1\Sha,df_2\Sha) 
\\
&  \nonumber
 + h_1\, (E{\tecka}h_2 + \Lambda(L_E\omega,df_2)) 
 - h_2\, (E{\tecka}h_1 + \Lambda(L_E\omega,df_1 ))
\big)\, E\,.
\end{align*} 
This vector field is the pre-Hamiltonian lift \eqref{Eq: 2.1} of the pair
\begin{multline*}
\big(
\{f_1,f_2\} ; 
 \{f_1,h_2\} - \{f_2,h_1\} - d\omega(df_1\Sha,df_2\Sha) 
\\
 + h_1\, (E{\tecka}h_2 + \Lambda(L_E\omega,df_2)) 
 - h_2\, (E{\tecka}h_1 + \Lambda(L_E\omega,df_1 ))
\big)\,.
\end{multline*}
Moreover, the above pair is  in $C^\infty_E(\f M) \times  C^\infty(\f M)$. Really,
\begin{align*}
E{\tecka}\{f_{1},f_{2}\}
& = 
\{E{\tecka}f_{1},f_{2}\} + \{f_{1},E{\tecka}f_{2}\} + i_{[E,\Lambda]} (df_{1} \wedge df_{2}) 
\\
& = \nonumber
-i_{E\wedge(L_E\omega)\Sha}(df_1 \wedge df_2) = 0\,
\end{align*}
which proves Theorem \ref{Th: 2.4}.
\end{proof}

As a consequence, by observing $\Lambda(L_E\omega,df) = (L_E\omega)\Sha {\tecka} f$, we obtained  the Lie bracket
\begin{align}\nonumber
\db[(f_{1},h_{1}) 
&
;(f_{2},h_{2})\db]
 =
\big(
\{f_{1},f_{2}\};
\\ 
& \label{Eq: 2.9}
\{f_{1},h_{2}\} - \{f_{2},h_{1}\} 
- d\omega(df_{1}\Sha,df_{2}\Sha)
\\
& \nonumber
+ h_1\, (E{\tecka}h_2 + (L_E\omega)\Sha{\tecka}f_2 ))
- h_2\, (E{\tecka}h_1 + (L_E\omega)\Sha{\tecka}f_1 ))
\big)\,
\end{align}
of pairs of functions from $ C_E^\infty(\f M)\times  C^\infty(\f M)$ which defines the local Lie algebra structure on $C^\infty_E(\f M) \times  C^\infty(\f M)$.

\smallskip

In \cite{Jan16} it was proved that all infinitesimal symmetries of $ \Omega
 $ are vector fields $ X_{(\alpha,h)} = \alpha\Sha + h\, E $, where $ \alpha $ is a closed 1-form such that $ \alpha(E) = 0 $. Hence, locally, $ \alpha = df $ for a function $ f \in C^{\infty}_E(\f M) $
 and any infinitesimal symmetry of $\Omega$ is locally the pre-Hamiltonian lift of a pair of functions from $C^\infty_E(\f M) \times  C^\infty(\f M)$\,.
So the Lie algebra of local generators of infinitesimal symmetries of $\Omega$ is $(\LGen(\Omega),\db[,\db]) \equiv (C^\infty_E(\f M) \times  C^\infty(\f M),\db[,\db])$ with the Lie bracket \eqref{Eq: 2.9}. 
The pre-Hamiltonian lift \eqref{Eq: 2.1} is the Lie algebra homomorphism from the local Lie algebra $ (\LGen(\Omega); \db[,\db]) $ to the Lie algebra $(\mathcal{L}(\Omega); [,])\subset (\mathcal{X}(\f M); [,])$ of infinitesimal symmetries of $\Omega$.

\begin{thm}\label{Th: 2.5}
A vector field $X_{(f,h)}$ is an infinitesimal symmetry of the almost-cosymplectic-contact structure $(\ome,\Ome)$
if and only if
 $ f \in C^\infty_E(\f M) $  and the condition \eqref{Eq: 2.5} is satisfied.
\end{thm}

\begin{proof}
It follows from Theorems \ref{Th: 2.2} and \ref{Th: 2.4}.
\end{proof}

\begin{lem}\label{Lm: 2.6}
A vector field $X_{(f,h)}$ is an infinitesimal symmetry of $(\omega,\Omega)$ if and only if
the following conditions are satisfied 
\begin{enumerate}
\item  $E\tecka f = i_E df = 0$\,, 
\item  $i_E dh + i_E i_{df\Sha}d\omega
= E{\tecka}h + (L_E\ome)\Sha{\tecka} f = 0$\,,
\item $ d\omega(df\Sha,\beta\Sha) + h \, d\omega(E,\beta\Sha) + dh(\beta\Sha) = 0$ for any 1-form $ \beta $, especially, if we put $\beta = dg$ for a $g\in C^\infty(\f M)$, we get
\\
$\{g,h\} = d\omega(dg\Sha,df\Sha) + h \, d\omega(dg\Sha,E)$\,.
\end{enumerate}
\end{lem}

\begin{proof}
It is a consequence of  Theorem \ref{Th: 2.5}  and Theorem \ref{Th: 2.2} where 
we have evaluated the 1-form on the left hand side of \eqref{Eq: 2.5} on $ E $ (which gives the condition (2)) and on $ \beta\Sha $ for any 1-form $ \beta $ (which gives the condition (3)).
\end{proof}

We denote the sheaf of pairs of functions which locally generate infinitesimal symmetries of  the almost-cosymplectic-contact structure $(\ome,\Ome)$
as $ \LGen(\omega,\Omega) = \LGen(\omega) \cap \LGen(\Omega)  $.
The brackets \eqref{Eq: 2.8} and \eqref{Eq: 2.9} restricted for generators of infinitesimal symmetries of $(\omega,\Omega)$ give the equivalet expressions of the bracket
\begin{align}\nonumber
\db[(f_{1},h_{1})  & ;(f_{2},h_{2})\db] =
\\ \label{Eq: 2.10}
& 
 =
\big(
\{f_{1},f_{2}\};\{f_{1},h_{2}\} - \{f_{2},h_{1}\} 
- d\omega(df_{1}\Sha,df_{2}\Sha)
\big)
\\ 
& \nonumber
 =
\big(
\{f_{1},f_{2}\};
d\omega(df_{1}\Sha,df_{2}\Sha)
+ h_2 \, (L_E\omega)\Sha{\tecka}f_{1} 
- h_1 \, (L_E\omega)\Sha{\tecka}f_{2} 
\big)
\\ 
& \nonumber
 =
\big(
\{f_{1},f_{2}\};
d\omega(df_{1}\Sha,df_{2}\Sha)
+ h_1 \, E{\tecka}h_{2} 
- h_2 \, E{\tecka}h_{1} 
\big)\,
\end{align}
which defines the Lie algebra structure on $\LGen(\omega,\Omega)$.

\begin{crl}\label{Cr: 2.7}
An infinitesimal symmetry
of the cosymplectic  structure $(\omega,\Omega)$ is of local type
$
X_{(f,h)} = df\Sha + h\, E\,,
$
where $ f \in C^\infty_E(\f M) $ and $h$ is a constant. 

Then the bracket \eqref{Eq: 2.10} is reduced to
\begin{equation*}
\db[(f_{1},h_{1});(f_{2},h_{2})\db]
=
\big(
\{f_{1},f_{2}\}, 0
\big)\,.
\end{equation*}
I.e. we obtain the subalgebra $ (\LGen_{cos}(\omega,\Omega);\db[,\db]) = (C^\infty_E(\f M);\{,\}) \oplus (\B R;[,]) $ of local generators of infinitesimal symmetries of the cosymplectic structure.
Here $[,]$  is the trivial Lie bracket in $\B R$.   
\end{crl}

\begin{proof}
It follows from $d\omega = 0$, then,  from \eqref{Eq: 2.5}, $dh = 0$.
\end{proof}

\begin{crl}\label{Cr: 2.8}
Any infinitesimal symmetry
of the contact  structure $(\omega,\Omega= d\omega)$ is
 of local type
\begin{equation*}
X_{(f,-f)} =  df\Sha - f\,E\,,
\end{equation*} 
where $ f \in C_E^\infty(\f M) $.

Then the bracket \eqref{Eq: 2.10} is reduced to
\begin{align*}
\db[(f_{1},-f_{1});(f_{2},-f_{2})\db]
& =
\big(
\{f_{1},f_{2}\}, - \{f_{1},f_{2}\}
\big)
\,.
\end{align*}
I.e. we obtain the subalgebra $ (\LGen_{con}(\omega);\db[,\db]) \subset (\LGen(\omega,\Omega);\db[,\db]) $ of local generators of infinitesimal symmetries of the contact structure. 
Moreover,
$
(\LGen_{con}(\omega); \db[ ; \db]) \equiv  (C^\infty_E(\f M), \{,\})\,. 
$
\end{crl}

\begin{proof}
It follows from $d\omega = \Omega$, $i_E\Omega = 0$ and $i_{df\Sha}\Omega = df$. Then, from
\eqref{Eq: 2.5}, $df = - dh$.
\end{proof}

%\begin{rmk}
%{\rm
%Let us consider Abelian subalgebra $(\M A(\f M);\db[,\db])\subset (\LGen(\omega);\db[,\db])$ given by %pairs $(c,0)\,,\,\, c\in \B R$. Then the centralizer of $\M A(\f M)$ in $(\LGen(\omega);\db[,\db])$ is %$\LGen(\omega,\Omega)$. Really, we have
%$$
%\db[ (c,0);(f,h)\db] 
%=
%(c(E\tecka f); 0) 
%=
%(0;0)
%$$  
%if and only if $E\tecka f = 0$ and the pair $(f,h)\in\LGen(\omega,\Omega)$.
%}
%\end{rmk}

%--------------------------------------------------------------------%
\subsection{Infinitesimal symmetries of $E $ and $\Lambda$ generated by pairs of functions}
\label{Sec: E pair of functions}
%\setcounter{equation}{0}
%--------------------------------------------------------------------%

\begin{thm}\label{Th: 2.9}
The pre-Hamiltonian lift \eqref{Eq: 2.1} is an infinitesimal symmetry of the Reeb vector field $ E $ if and only if 
\begin{align}\label{Eq: 2.11}
\big(
d(E \tecka f) - (E\tecka f)L_E\omega
\big)\Sha
& = 
0\,,
\\
(E\tecka h) + (L_E\omega)\Sha \tecka f \label{Eq: 2.12}
&=
0
\,.
\end{align}
\end{thm}

\begin{proof}
Let us assume that the pre-Hamiltonian lift of a pair $(f,h)$ is an infinitesimal symmetry of the Reeb vector field. Then from \eqref{Eq: 2.3} 
\begin{align*}
0 = [X_{(f,h)},E]  = - \big(
d(E \tecka f) - (E\tecka f)L_E\omega
\big)\Sha - ((E\tecka h) + (L_E\omega)\Sha \tecka f)\, E\,
\end{align*}
which proves Theorem \ref{Th: 2.9}.
\end{proof}

If we have two pairs $(f_i,h_i)$  generating infinitesimal symmetries of $E$ then the Lie bracket of the corresponding pre-Hamiltonian lifts is  
\begin{align}\label{Eq: 2.13}
& [X_{(f_1,h_1)},  X_{(f_2,h_2)}]
 =
\big(
d\{f_1,f_2\} + (E{\tecka}f_1)\,L_{df_2\Sha}\omega
- (E{\tecka}f_2)\,L_{df_1\Sha}\omega
\big)\Sha
\\
& \quad\nonumber
+ \big(
 \{f_1,h_2\} - \{f_2,h_1\} - d\omega(df_1\Sha,df_2\Sha) 
\big)\, E\,.
\end{align}
The above vector field \eqref{Eq: 2.13} is not generally the pre-Hamiltonian lift of a pair of functions. So, the sheaf of local generators of infinitesimal symmetries of $E$ is not a Lie algebra.

\begin{thm}\label{Th: 2.10}
The pre-Hamiltonian lift \eqref{Eq: 2.1} is an infinitesimal symmetry of
 $ \Lambda $, i.e. $ L_{X_{(f,h)}} \Lambda = [X_{X_{(f,h)}},\Lambda] = 0 $, if and only if  the following condition is satisfied
\begin{equation}\label{Eq: 2.14}
[df\Sha,\Lambda] - E \wedge (dh + h\, L_E\omega)\Sha = 0\,.
\end{equation}
\end{thm}

\begin{proof}
We have
$$
L_{X_{(f,h)}} \Lambda = [df\Sha,\Lambda] + [h\, E,\Lambda]\,. 
$$
Theorem \ref{Th: 2.10} follows from
$$
[h\, E,\Lambda] = h \, [E,\Lambda] - E\wedge dh\Sha = - E\wedge (dh + h\, L_E\omega)\Sha\,.
$$
\vglue-1.3\baselineskip
\end{proof}

\begin{lem}\label{Lm: 2.11}
A vector field $ X_{(f,h)} $ is an infinitesimal symmetry of $ \Lambda $
if and only if the following conditions
\begin{align}\label{Eq: 2.15}
E\tecka f 
& =
0\,,
\\
d\omega(df\Sha,\beta\Sha) + h \, d\omega(E,\beta\Sha) + dh(\beta\Sha)  
& = \label{Eq: 2.16}
0\,,
\end{align}
are satisfied for any 1-form $ \beta$.
\end{lem}

\begin{proof}
It is sufficient to evaluate the 2-vector field on the left hand side of \eqref{Eq: 2.14} on $\ome, \beta$ and $\beta,\gamma$, where $\beta, \gamma$ are closed 1-forms. We get
$$
i_{[df\Sha,\Lambda] - E \wedge (dh + h\, L_E\omega)\Sha}(\omega \wedge \beta) = - \Lambda(i_{df\Sha}d\omega + h\, L_E\omega + dh,\beta)
$$
which vanishes if and only if \eqref{Eq: 2.16} is satisfied.

On the other hand
\begin{align*}
&
i_{[df\Sha,\Lambda] 
 - 
E \wedge (dh + h\, L_E\omega)\Sha} 
(\beta \wedge \gamma) = 
\\
& \qquad\qquad =
\Lambda(df,d\Lambda(\beta,\gamma)) 
+ \Lambda(\beta,d\Lambda(\gamma,df))
+ \Lambda(\gamma,d\Lambda(df,\beta)) 
\\
& \qquad\qquad\quad
- \beta(E)\,\Lambda(h\, L_E\omega + dh,\gamma)
+ \gamma(E)\,\Lambda(h\, L_E\omega + dh,\beta)
\end{align*}
which, by using \eqref{Eq: 2.16}, can be rewritten as
\begin{align*}
&
i_{[df\Sha,\Lambda] 
 - 
E \wedge (dh + h\, L_E\omega)\Sha} 
(\beta \wedge \gamma) = - \tfrac 12 i_{[\Lambda,\Lambda]} (df\wedge\beta\wedge\gamma)
\\
& \qquad\quad
+ \beta(E)\,\Lambda(i_{df\Sha}d\omega,\gamma)
- \gamma(E)\,\Lambda(i_{df\Sha}d\omega,\beta)
\\
& \qquad =
-  i_{E\wedge (\Lambda\Sha\otimes\Lambda\Sha)(d\omega)} (df\wedge\beta\wedge\gamma)
\\
& \qquad\quad+ \beta(E)\,\Lambda(i_{df\Sha}d\omega,\gamma)
- \gamma(E)\,\Lambda(i_{df\Sha}d\omega,\beta)
\\
& \qquad =
- df(E)\, d\omega(\beta\Sha,\gamma\Sha) 
=
- (E \tecka f)\, d\omega(\beta\Sha,\gamma\Sha)
\end{align*}
which vanishes if and only if \eqref{Eq: 2.15} is satisfied.

On the other hand if \eqref{Eq: 2.15} and \eqref{Eq: 2.16} are satisfied, then the 2-vector field 
$L_{X_{(f,h)}}\Lambda $ is the zero 2-vector field.
\end{proof}

So, local generators of infinitesimal symmetries of $\Lambda$ are from $\LGen(\Omega)$ and generates also infinitesimal symmetries of $\Omega$, i.e.
$\LGen(\Lambda) \subset \LGen(\Omega)$ and the Lie algebra of local generators of infinitesimal symmetries of $ \Lambda $ is the subalgebra $(\LGen(\Lambda); \db[ , \db]) \subset (\LGen(\Omega); \db[ , \db]) $ with the bracket
 \begin{align*}%\label{Eq: 2.17}
\db[(f_{1},h_{1})  & ;(f_{2},h_{2})\db] =
\\ 
& 
 = \nonumber
 \big(
\{f_{1},f_{2}\};
d\omega(df_{1}\Sha,df_{2}\Sha)
+ h_1 \, E{\tecka}h_{2} 
- h_2 \, E{\tecka}h_{1} 
\big)
\,
\end{align*}
which we obtain from the bracket \eqref{Eq: 2.9} restricted for functions satisfying \eqref{Eq: 2.16}.

\begin{crl}\label{Cr: 2.12}
The Lie algebra of local generators of infinitesimal symmetries of the almost-coPoisson-Jacobi pair $(E,\Lam)$ coincides with the local Lie algebra of generators of infinitesimal symmetries of $(\omega,\Omega)$. 
\end{crl}

\begin{proof}
From Theorem \ref{Th: 2.9} and Lemma \ref{Lm: 2.11} the pre-Hamiltonian lift \eqref{Eq: 2.1} is an infinitesimal symmetry of the pair
 $(E, \Lambda) $ if and only if the conditions (1), (2) and (3) of Lemma \ref{Lm: 2.6} are satisfied, i.e. if and only if it is an infinitesimal symmetry of the pair $(\omega,\Omega)$.
\end{proof}

\begin{rmk}\label{Rm: 2.1}
{\rm
Let us consider subsheaves of $ C^{\infty}(\f M) \times C^{\infty}(\f M) $ given by conditions (1), (2) and (3) of Lemma \ref{Lm: 2.6}.
The subsheaf $ \LGen(\Omega) $ is given by the condition (1), the subsheaf $ \LGen(\omega) $ is given by the conditions (2) and (3) and the subsheaf $ \LGen(\Lambda) $ is given by the conditions (1) and (3). If we assume the subsheaf given by conditions 
(1) and (2) then, from Theorem \ref{Th: 2.9}, it is the subsheaf $ \LGen(E,\Omega) $ of local generators of infinitesimal symmetries of the Reeb vector field $E$ and $\Omega$. 
The corresponding bracket will be
\begin{align*}%\label{Eq: 2.18}
\db[(f_{1},h_{1})  & ;(f_{2},h_{2})\db] =
\\ 
& 
 = \nonumber
\big(
\{f_{1},f_{2}\};\{f_{1},h_{2}\} - \{f_{2},h_{1}\} 
- d\omega(df_{1}\Sha,df_{2}\Sha)
\big)\,.
\end{align*}
So we obtain the following local Lie algebras of generators of infinitesimal symmetries

\bigskip
\begin{center}
\begin{tabular}{||l|l||}
\hline\hline
$ (\LGen(\Ome),\db[,\db]) $ & (1), $ E\tecka f = 0 $
\\
\hline
$ (\LGen(\omega),\db[,\db]) $ & (2), $ E\tecka h + (L_E\omega)\Sha\tecka f = 0 $
\\
 & (3), $ d\omega(df\Sha,\beta\Sha) + h \, d\omega(E,\beta\Sha) + dh(\beta\Sha) = 0$
\\ 
 \hline
$ (\LGen(\Lambda),\db[,\db]) $ & (1) and (3)
\\
\hline
$ (\LGen(E,\Ome),\db[,\db]) $ & (1) and (2)
\\
\hline
$ (\LGen(\omega,\Ome),\db[,\db]) $ &
\\
$ \equiv  (\LGen(E,\Lambda),\db[,\db])$ & (1), (2) and (3)
\\
\hline \hline
\end{tabular}
\end{center}

}{ }\hfill\END
\end{rmk}

\begin{rmk}\label{Rm: 2.2}
{\rm
In the Lie algebra $ (\LGen(\Omega); \db[,\db]) $ we have the Abelian subalgebra formed by pairs of constant functions $\M K(\f M) = \B R \times \B R \subset  \LGen(\Omega) $. 
The centralizer of the Lie subalgebra $\M K(\f M)$ in $ \LGen(\Omega) $
is the Lie subalgebra $ (\LGen(E,\Omega); \db[,\db]) \subset  (\LGen(\Omega); \db[,\db])  $ of generators of infinitesimal symmetries of $E$ and $\Omega$. 

Really, let $(c,k) \in \M K(\f M)$ and $(f,h) \in \LGen(\Omega)$ such that
$$
(0;0) 
=
\db[
(c,k);(f,h)
\db]
= 
\big(0; - k\,(E{\tecka}h + (L_E\omega)\Sha{\tecka} f)
\big)\,.
$$
Then, from Theorem \ref{Th: 2.9}, it follows that the pair $ (f,h) $
generates an infinitesimal symmetry of  $E$. 
}\hfill\END
\end{rmk}

%--------------------------------------------------------------------%
\subsection{Multiplicative algebra $ (\LGen(\Omega), \cdot) $}
%\setcounter{equation}{0}
%--------------------------------------------------------------------%

In Section \ref{Sec: omega pais of functions}  we have defined the local Lie algebra structure on  $\LGen(\Omega)$. 
 We define the multiplication in $\LGen(\Omega)$ by
\begin{equation*}%\label{Eq: 2.19}
(f_1,h_1)(f_2,h_2) = (f_1\, f_2,f_1\, h_2 + f_2\, h_1)
\end{equation*}
which defines on $\LGen(\Omega)$ the structure of an associative commutative algebra with the unit $(1,0)$. Really it is easy to see that
$$
(f_1,h_1)(f_2,h_2) = (f_2,h_2)(f_1,h_1)\,,
$$
$$
\big((f_1,h_1)(f_2,h_2)\big)(f_3,h_3)
=
(f_1,h_1)\big((f_2,h_2)(f_3,h_3)\big)\,,
$$
$$
(1,0)(f,h) = (f,h)\,. 
$$

\begin{lem}\label{Lm: 2.13}
We have 
\begin{equation*}%\label{Eq: 2.20}
X_{(f_1\,f_2, f_1\,h_2 + f_2 \, h_1)}
=
f_1 \, X_{(f_2,h_2)}
+ f_2 \, X_{(f_1,h_1)}\,.
\end{equation*}
\end{lem}

\begin{proof}
We have
\begin{align*}
X_{(f_1\,f_2, f_1\,h_2 + f_2 \, h_1)}
&= 
d(f_1\, f_2)\Sha + (f_1\,h_2 + f_2 \, h_1)\, E
\\
& =
f_1 \, (df_2\Sha + h_2 \, E) + 
f_2 \, (df_1\Sha + h_1 \, E)
\\ 
& =
 f_1 \, X_{(f_2,h_2)}
+ f_2 \, X_{(f_1,h_1)}\,.
\end{align*}
\vglue-1.1\baselineskip
\end{proof}

From \eqref{Eq: 2.9} it is easy to see that the bidifferential operator $D$ on $\LGen(\Omega)$ given by
$$
D_{(f_1,h_1)}(f_2,h_2) = \db[ (f_1,h_1); (f_2,h_2) \db] 
$$
is of order 1. 

From the Jacobi identity we get
\begin{multline*}
D_{\db[(f_1,h_1);(f_2,h_2)\db]}(f_3,h_3) =
\big(D_{(f_1,h_1)} D_{(f_2,h_2)} - D_{(f_2,h_2)} D_{(f_1,h_1)}
\big)\, (f_3,h_3)
\,
\end{multline*}
and 
\begin{multline*}
D_{(f_1,h_1)}\db[(f_2,h_2);(f_3,h_3)\db] =
\\
=
\db[D_{(f_1,h_1)}(f_2,h_2);(f_3,h_3\db]) + \db[(f_2,h_2); D_{(f_1,h_1)}(f_3,h_3)\db]
\,,
\end{multline*}
i.e., $D_{(f,h)}$ is a derivation on $(\LGen(\Omega); \db[,\db])$.

\begin{thm}\label{Th: 2.14}
The 1st order differential operator
$$
D_{(f,h)}:\LGen(\Omega)\to \LGen(\Omega)
$$
is a derivation on $ (\LGen(\Omega), \cdot ) $, i.e. for $(f_i,h_i)\in \LGen(\Omega)$, $ i = 1,2,3 $, we have
\begin{multline*}%\label{Eq: 2.21}
D_{(f_1,h_1)}\big((f_2,h_2)(f_3,h_3)\big) =
(f_2,h_2) D_{(f_1,h_1)}(f_3,h_3)
\\
+
(f_3,h_3) D_{(f_1,h_1)}(f_2,h_2)\,.
$$
\end{multline*}
\end{thm}

\begin{proof}
From \eqref{Eq: 2.9} we can prove
\begin{multline*}%\label{Eq: 2.22}
\db[(f_1,h_1);(f_2,h_2)(f_3, h_3) \db]
 =
(f_2,h_2)\, \db[(f_1,h_1);(f_3, h_3) \db]
\\
+ (f_3,h_3) \db[(f_1,h_1);(f_2, h_2) \db]
 \,
\end{multline*}
which proves Theorem \ref{Th: 2.14}.
\end{proof}

%--------------------------------------------------------------------%
\subsection{Lie derivation of pairs of functions}
%\setcounter{equation}{0}
%--------------------------------------------------------------------%
We define the Lie derivation of pairs of functions $(f,h)\in C^\infty(\f M)\times C^\infty(\f M)$ given by a vector field $X$ on $\f M$ by $L_X(f,h) = (L_Xf,L_Xh) = (X{\tecka}f,X{\tecka}h) $.
Generally $ L_X $ is not an operator on $ \LGen(\Omega) $. 

\begin{lem}\label{Lm: 2.15}
Let $X$ be a vector field on $\f M$ such that $L_XE = [X,E] = 0$.
Then for $(f,h) \in \LGen(\Omega)$ the Lie derivation
$L_X(f,h)\in \LGen(\Omega)$.
\end{lem}

\begin{proof}
We have to prove that $L_Xf\in C^\infty_E(\f M)$. 
But from
$$
0 = L_{[X,E]}f = L_XL_E f - L_EL_X f = - E{\tecka}(L_X f)\,,
$$
hence $L_X(f,h)\in \LGen(\Omega)$.
\end{proof}

\begin{lem}\label{Lm: 2.16}
If a vector field   $ X$ is an infinitesimal symmetry of $ E $ then
$L_X$ is a derivation on $(\LGen(\Omega), \cdot )$, i.e.
$$
L_X\big((f_1,h_1)(f_2,h_2)\big)
=
(L_Xf_1,L_Xh_1) (f_2,h_2) 
+
(f_1,h_1)  (L_Xf_2,L_Xh_2)\,.
$$
\end{lem}

\begin{proof}
By Lemma \ref{Lm: 2.15} $ L_X\big((f_1,h_1)(f_2,h_2)\big) \in \LGen(\Omega)$ and
\begin{align*}
L_X\big((f_1,h_1) & (f_2,h_2)\big)
  =
L_X (f_1\, f_2, f_1\, h_2 + f_2\, h_1)
\\
& =
\big( (L_Xf_1)\, f_2 + f_1\,(L_Xf_2);
\\
& \quad (L_Xf_1)\, h_2 + f_1\,(L_Xh_2)
+ (L_Xf_2)\, h_1 + f_2\,(L_Xh_1)
\big)
\\
& =
(L_Xf_1, L_Xh_1) (f_2, h_2) + (f_1, h_1) (L_Xf_2, L_Xh_2)\,.
\end{align*} 
\vglue-1.3\baselineskip
\end{proof}

On the other hand it is easy to see that if $ X $ is an infinitesimal symmetry of $ E $, then $ L_X $ is not  a derivation on the Lie algebra $(\LGen(\Omega), \db[,\db])$. But we have

\begin{thm}\label{Th: 2.17}
Let $X$ be an infinitesimal symmetry of the
almost-coPoisson-Jacobi structure $(E,\Lambda)$, then $X$ is a Lie derivation on the Lie algebra
$(\LGen(\omega,\Omega); \db[ , \db ])$.
\end{thm}

\begin{proof}
 $L_X E = 0$ and $L_X\Lambda =0$ imply 
$$
L_X\{ f_1,f_2 \} = \{ L_Xf_1,f_2 \} + \{ f_1, L_Xf_2 \}\,,
$$
$L_EL_X\omega =0$, $L_EL_X f = L_XL_Ef$ and $L_Xd\omega=0$.

We have to prove that $ (L_Xf,L_Xh) \in \LGen(\omega,\Omega) $ for any $ (f,h) \in \LGen(\omega,\Omega) $. First, from Lemma \ref{Lm: 2.15}, $ L_Xf \in C_E^\infty(\f M)$. Further, we have to prove conditions (2) and (3) of Lemma \ref{Lm: 2.6} for the pair of functions $ (L_X f, L_X h) $.
The condition (2) can be expressed as
\begin{equation*}
0 = dh(E) + \Lambda(L_E\omega,df)\,.
\end{equation*}
If we apply $ L_X  $ on the above identity we get, from $ L_X df = dL_X f $,
\begin{align*}
0
& =
(L_Xdh)(E) + dh(L_XE) + (L_X\Lambda)(L_E\omega,df)
\\
&\quad
+ \Lambda(L_XL_E\omega,df) + \Lambda(L_E\omega,L_Xdf)
\\
& =
(dL_X h) (E) + \Lambda(L_E\omega,dL_Xf)
\end{align*}
which is the condition (2) for the pair  $ (L_X f, L_X h) $.
 
Further, applying $L_X$ on (3) we get 
\begin{align*}
0
& =
(L_Xd\omega)(df\Sha, \beta\Sha) + d\omega(L_Xdf\Sha,\beta\Sha)
+ d\omega(df\Sha,L_X\beta\Sha)
\\
& \quad
+ (L_Xh)\, d\omega(E,\beta\Sha) + h\, d\omega(L_XE,\beta\Sha)
+ h\, d\omega(E,L_X\beta\Sha)
\\
& \quad
+ (L_Xdh)(\beta\Sha) + dh(L_X\beta\Sha)
\\
&
=
d\omega(d(L_Xf)\Sha,\beta\Sha)
+ (L_Xh)\, d\omega(E,\beta\Sha)
+ (dL_Xh)(\beta\Sha)
\end{align*}
which follows from $L_X df\Sha = d(L_Xf)\Sha$ (see \cite{Jan15}, Lemma 2.15), and the condition (3) 
for the pair  $ (L_X f, L_X h) $ is satisfied.
Hence $ (L_X f, L_X h)\in \LGen(\omega,\Omega). $

Finally, we assume the bracket \eqref{Eq: 2.10}
 in the form
\begin{equation*}
\db[  (f_1,h_1);(f_2,h_2)\db]
=
\big(\{ f_1,f_2\} ; d\omega(df_1\Sha,df_2\Sha) 
+ h_1 \, E\tecka h_2 - h_2 \, E\tecka h_1
\big)\,.
\end{equation*} 
Then
\begin{align*}
L_X\db[ & (f_1,h_1);(f_2,h_2)\db] 
= 
\big(\{ L_X f_1,f_2\} + \{ f_1,L_X f_2\} ;
\\
& \qquad
d\omega(L_Xdf_1\Sha,df_2\Sha)
+ d\omega(df_1\Sha,L_X df_2\Sha)
\\
& \qquad
+ L_X h_1\, L_Eh_2 + h_1\, L_XL_Eh_2
- L_X h_2\, L_Eh_1
- h_2\, L_XL_Eh_1
\big)\,.
\end{align*}
On the other hand 
\begin{multline*}
\db[ (L_Xf_1,L_Xh_1);(f_2,h_2)\db] 
+ \db[(f_1,h_1) ; (L_Xf_2,L_Xh_2)\db]
=
\\
= \big(\{ L_X f_1,f_2\} ; d\omega(d(L_Xf_1)\Sha,df_2\Sha) 
\\
{} \quad
+ L_X h_1\, L_Eh_2 - h_2\, L_EL_Xh_1 
\big)
\\
+ \big(
\{ f_1,L_X f_2\} ;   d\omega(df_1\Sha,d(L_X f_2)\Sha)
\\
{} \quad
+ h_1\, L_E L_Xh_2 - L_X h_2 \, L_Eh_1 
\big)
\end{multline*}
and  from  $L_X df\Sha = d(L_Xf)\Sha$
we get
\begin{align*}
L_X\db[  (f_1,h_1) 
& ;
(f_2,h_2)\db] 
= 
\\
& =
\db[ (L_Xf_1,L_Xh_1);(f_2,h_2)\db] 
+ \db[(f_1,h_1) ; (L_Xf_2,L_Xh_2)\db]\,,
\end{align*}
i.e. $ L_X $ is a derivation on the Lie algebra
$ (\LGen(\omega,\Omega); \db[ , \db ]). $
\end{proof}

\begin{rmk}\label{Rm: 2.3}
{\rm
We have
\begin{equation}\label{Eq: 2.23}
\db[  (f_1,h_1) ;(f_2,h_2)\db]
 =
\frac 12 \big( L_{X_{(f_1,h_1)}}(f_2,h_2)
-
L_{X_{(f_2,h_2)}}(f_1,h_1)\big)\,.
\end{equation}
Really
\begin{align*}
L_{X_{(f_1,h_1)}}(f_2,h_2)
& -
L_{X_{(f_2,h_2)}}(f_1,h_1)
=
\\
& = 
\big(
2\, \{f_1,f_2); \{f_1,h_2\} - \{f_2,h_1\} 
+ h_1 \, E\tecka h_2 - h_2 \, E\tecka h_1
\big)\,
\end{align*}
and from (2) and (3) of Lemma \ref{Lm: 2.6} we have 
\begin{align*}
\{f_1,h_2\} - \{f_2,h_1\}
= 2\, d\omega(df_1\Sha,df_2\Sha)+ h_1 \, E\tecka h_2 - h_2 \, E\tecka h_1
\end{align*}
which implies \eqref{Eq: 2.23}.
\hfill\END
}
\end{rmk}

%--------------------------------------------------------------------%
\section{Lie algebroid and infinitesimal symmetries}
\setcounter{equation}{0}
\label{Sec: LA}
%--------------------------------------------------------------------%

Let us assume a closed 2--form $F$ on $\f M$ and the vector bundle
$\f E = T\f M \oplus_{\f M} \Bbb R \to \f M$. Then sections of $\f E$ are pairs $(X,f)$ of vector fields on $\f M$ and functions on $\f M$. We define a bracket
of sections of $\f E$ by
\begin{equation}\label{Eq: 3.1}
\db[(X_1,f_1);(X_2,f_2) \db]_F
=
\big(
[X_1, X_2] ; X_1.f_2 - X_2.f_1 + F(X_1, X_2)
\big)\,.
\end{equation}
This bracket defines an $ F $-Lie algebroid structure (see, for instance, \cite{Mac05}) on $\f E$ with the anchor given by the projection on the first component. 

Really, the bracket \eqref{Eq: 3.1} is antisymmetric and from the closure of $F$ the Jacobi identity is satisfied. Moreover,
for any $h \in C^\infty(\f M)$ we get
\begin{align*}
\db[(X_1,f_1);h\,(X_2,f_2) \db]_F
& =
\db[(X_1,f_1);(h\,X_2,h\,f_2) \db]_F
\\
& =
h\, \db[(X_1,f_1);(X_2,f_2) \db]_F + (X_1.h)\, (X_2,f_2)\,
\end{align*}
and the Leibniz-type formula is satisfied. 

%--------------------------------------------------------------------%
%\subsection{Lie algebroid and infinitesimal symmetries of $ (\omega,\Omega) $ }
%\setcounter{equation}{0}
%--------------------------------------------------------------------%
 
 \medskip
 Now, let us assume the sheaf mapping from the local Lie algebra $ (\LGen(\omega,\Omega); \db[ , \db]) $ to sections of the $ F $-Lie algebroid given by pairs of vector fields on $ \f M $ and functions on $ \f M $  given by
\begin{equation}\label{Eq: 3.2}
\f s:(f,h) \mapsto (X_{(f,h)}, f-h)\,.
\end{equation}   

\begin{thm}\label{Th: 3.1}
For the closed 2-form
\begin{equation*}%\label{Eq: 3.3}
F = \Omega + d\omega\,
\end{equation*} 
the sheaf mapping
\begin{align*}
\f s:\LGen(\omega,\Omega)  
& 
\to \M X(\f M) \times  C^\infty(\f M)
\end{align*}
given by \eqref{Eq: 3.2} 
is a Lie algebra morphism. 
\end{thm}

\begin{proof}
For $(f_i,h_i)\in \LGen(\omega,\Omega)$, $i=1,2,$
we have
\begin{multline*}
\db[
(f_1,h_1);(f_2,h_2)
\db]
=
\big(
\{f_1,f_2 \};
\{f_1,h_2 \} - \{f_2,h_1 \} - d\omega(df_1\Sha,df_2\Sha)
\big) 
\\
\mapsto
\big(d\{f_1,f_2 \}\Sha
+
\big(
\{f_1,h_2 \} - \{f_2,h_1 \} - d\omega(df_1\Sha,df_2\Sha)
\big) \, E;
\\
\{f_1,f_2 \} - \{f_1,h_2 \} + \{f_2,h_1 \} + d\omega(df_1\Sha,df_2\Sha)
\big)\,.
\end{multline*} 

On the other hand we have
\begin{multline*}
\db[
\f s(f_1,h_1);\f s(f_2,h_2)
\db]_F
=
\\
= \big(
d\{f_1,f_2 \}\Sha
+
\big(
\{f_1,h_2 \} - \{f_2,h_1 \} - d\omega(df_1\Sha,df_2\Sha)
\big) \, E
;
\\
2\,\{f_1,f_2 \} - \{f_1,h_2 \} + \{f_2,h_1 \}
\\
- h_1 E{\tecka} h_2 + h_2 E{\tecka}h_1 +  F(df_1\Sha + h_1 E,df_2\Sha + h_2 E)
\big)
\,.
\end{multline*}
The first parts of the above pairs are equal. The second parts are equal if and only if
\begin{multline*}
0
 =
\{f_1,f_2 \} -  d\omega(df_1\Sha,df_2\Sha)
- h_1 E{\tecka} h_2 + h_2 E{\tecka}h_1 +  F(df_1\Sha + h_1 E,df_2\Sha + h_2 E)
\,
\end{multline*}
which can be rewritten, by using  $\Omega(df_1\Sha,df_2\Sha)= - \Lambda(df_1,df_2)= - \{f_1,f_2\}$ (see \cite{JanMod09}), as
\begin{multline*}
 0
 = 
- h_1 (E{\tecka} h_2 - F(E,df_2\Sha)) + h_2 (E{\tecka}h_1 - F(E,df_1\Sha)) 
\\
+ (F-d\omega - \Omega)(df_1\Sha,df_2\Sha)
\,.
\end{multline*}
From $\Omega(E,df_i\Sha) = 0$\,,  $E{\tecka}h_i - d\omega(E,df_i\Sha) = 0$ (see Lemma \ref{Lm: 2.6})
it is equivalent with
\begin{equation*}
0 = (F - d\omega - \Omega)(X_{(f_1,h_1)},X_{(f_2,h_2)})
\end{equation*}
which is satisfied for $F = \Omega + d\omega$.
\end{proof}

\begin{lem}\label{Lm: 3.2}
Let us consider the $F$-Lie algebroid given by the closed 2-form $F= \Omega + d\omega.$
Let $(X_i,\breve{f}_i)\in \M X(\f M) \times  C^\infty(\f M)$, $i= 1,2$, be pairs such that $X_i$ are infinitesimal symmetries of $(\omega,\Omega)$ and $E{\tecka}\breve{f}_i= - E{\tecka}(\omega(X_i))$. Then the bracket $\db[,\db]_F$ of these pairs satisfy the same conditions.
\end{lem}

\begin{proof}
$[X_1,X_2]$ is an infinitesimal symmetry of $(\omega,\Omega)$ so it is sufficient to prove that
$$
E{\tecka}\big(
X_1 \tecka \breve{f}_2 - X_2 \tecka \breve{f}_1
+ F(X_1,X_2)
\big)
= 
- E{\tecka}(\omega([X_1,X_2]))
$$
which can be rewritten as
\begin{multline*}
X_1 \tecka(E{\tecka} \breve{f}_2) - X_2 \tecka(E{\tecka} \breve{f}_1)
+ E{\tecka}(F(X_1,X_2))
=
\\
= 
- E{\tecka}(X_1\tecka (\omega(X_2)) - X_2\tecka (\omega(X_1) - d\omega(X_1,X_2))\,.
\end{multline*}
From the condition $E{\tecka}\breve{f}_i= - E{\tecka}(\omega(X_i))$ this equation is satisfied if and only if
$$
E{\tecka}(F(X_1,X_2)) = E{\tecka}(d\omega(X_1,X_2))\,
$$
which is satisfied for $F = \Omega + d\omega$ because of $L_E\Omega = 0$ and $L_EX_i = [E,X_i] = 0$. 
\end{proof}

\begin{thm}\label{Th: 3.3}
Let us consider the $F$-Lie algebroid given by the closed 2-form $F= \Omega + d\omega.$
Let $(X_i,\breve{f}_i)\in \M X(\f M) \times  C^\infty(\f M)$, $i= 1,2$, be pairs such that $X_i$ are infinitesimal symmetries of $(\omega,\Omega)$ and $E{\tecka}\breve{f}_i= - E{\tecka}(\omega(X_i))$. Then the sheaf mapping 
$$
\f r:\M X(\f M) \times  C^\infty(\f M)
\to
  C^{\infty}(\f M) \times  C^{\infty}(\f M)
$$
given by 
\begin{equation}\label{Eq: 3.4}
\f r:(X,\breve{f}) \mapsto (\omega(X) + \breve{f}, \omega(X))
\end{equation}
has values in $ \LGen(\omega,\Omega)  $ and it
is a Lie algebra morphism inverse to $\f s$.
\end{thm}

\begin{proof}
First we have to prove that the mapping \eqref{Eq: 3.4} has values in $ \LGen(\omega,\Omega)  $. Let us assume that $X$ is an infinitesimal symmetry of $(\omega,\Omega)$. By Theorem \ref{Th: 2.5} $X = df\Sha + h\, E$, where $E{\tecka}f =0$ and the condition \eqref{Eq: 2.5} is satisfied. It is easy to see that $h = \ome(X)$. Then 
$$
f = \ome(X) + \breve{f}
$$
is a conserved function.
 So we have
\begin{equation}\label{Eq: 3.5}
X = d(\omega(X) + \breve{f})\Sha + \omega(X)\, E\,
\end{equation}
and the pair $ (\omega(X) + \breve{f},\omega(X)) $ is in
$ \LGen(\omega,\Omega)  $.

Now, we have
\begin{multline*}
\db[
(X_1,\breve{f}_1 );(X_2,\breve{f}_2)
\db]_F 
\mapsto
\\
\mapsto
\big(
\omega([X_1,X_2]) + X_1\tecka \breve{f}_2 - X_2\tecka \breve{f}_1 + \Omega(X_1,X_2)
+ d\omega(X_1,X_2);
\\
\omega([X_1,X_2])
\big)\,.
\end{multline*}
On the other hand 
\begin{multline*}
\db[\f r
(X_1,\breve{f}_1 );\f r(X_2,\breve{f}_2)
\db] 
=
\big(
\{\omega(X_1) + \breve{f}_1 , \omega(X_2) + \breve{f}_2\};
\\
\{\omega(X_1) + \breve{f}_1 , \omega(X_2) \}
- \{\omega(X_2) + \breve{f}_2 , \omega(X_1) \}
\\
- d\omega(d(\omega(X_1) + \breve{f}_1)\Sha,d(\omega(X_2) + \breve{f}_2)\Sha)
\big)
\,.
\end{multline*}
These expressions are equal if and only if the following two equations are satisfied
\begin{multline}\label{Eq: 3.6}
\omega([X_1,X_2]) + X_1.\breve{f}_2 - X_1.\breve{f}_2 + \Omega(X_1,X_2)
+ d\omega(X_1,X_2) =
\\ 
\{\omega(X_1) + \breve{f}_1 , \omega(X_2) + \breve{f}_2\}\,,
\end{multline}
\begin{multline}\label{Eq: 3.7}
\omega([X_1,X_2])
=
\{\omega(X_1) + \breve{f}_1 , \omega(X_2) \}
- \{\omega(X_2) + \breve{f}_2 , \omega(X_1) \}
\\ 
- d\omega(d(\omega(X_1) + \breve{f}_1)\Sha,d(\omega(X_2) + \breve{f}_2)\Sha)\,.
\end{multline}
By using \eqref{Eq: 3.5} and $(\Lambda\Sha\otimes \Lambda\Sha)(\Ome) = - \Lambda$ we can rewrite the left hand side of \eqref{Eq: 3.6} as
\begin{multline*}
X_1.(\omega(X_2) + \breve{f}_2)
- X_2.(\omega(X_1) + \breve{f}_1)
+ \Omega(X_1,X_2) =
\\
= i_{d(\omega(X_1) + \breve{f}_1)\Sha + \omega(X_1)\, E}d(\omega(X_2) + \breve{f}_2) 
-
i_{d(\omega(X_2) + \breve{f}_2)\Sha + \omega(X_2)\, E}d(\omega(X_1) + \breve{f}_1) 
\\
+ \Omega(d(\omega(X_1) + \breve{f}_1)\Sha,d(\omega(X_2) + \breve{f}_2)\Sha)=
\\
=
\{\omega(X_1) + \breve{f}_1 , \omega(X_2) + \breve{f}_2\}
- \{\omega(X_2) + \breve{f}_2 , \omega(X_1) + \breve{f}_1\}
\\
- \{\omega(X_1) + \breve{f}_1 , \omega(X_2) + \breve{f}_2\} =
\\
=
\{\omega(X_1) + \breve{f}_1 , \omega(X_2) + \breve{f}_2\}\,.
\end{multline*}
Similarly, by using $i_Ed(\omega(X)) = d\omega(E,X)$, we can rewrite the left hand side of \eqref{Eq: 3.7} as
  \begin{multline*}
X_1.(\omega(X_2))
- X_2.(\omega(X_1))
- d\omega(X_1,X_2) =
\\
= i_{d(\omega(X_1) + \breve{f}_1)\Sha + \omega(X_1)\, E}d(\omega(X_2)) 
-
i_{d(\omega(X_2) + \breve{f}_2)\Sha + \omega(X_2)\, E}d(\omega(X_1))
\\
- d\omega(X_1,X_2) =
\\
 =\{\omega(X_1) + \breve{f}_1 , \omega(X_2) \}
- \{\omega(X_2) + \breve{f}_2 , \omega(X_1) \}
\\ 
- d\omega(X_1 - \omega(X_1) E ,X_2 - \omega(X_2) E)
\\
=
\{\omega(X_1) + \breve{f}_1 , \omega(X_2) \}
- \{\omega(X_2) + \breve{f}_2 , \omega(X_1) \}
\\ 
- d\omega(d(\omega(X_1) + \breve{f}_1)\Sha,d(\omega(X_2) + \breve{f}_2)\Sha)\,
\end{multline*}
which proves that \eqref{Eq: 3.4} is a Lie algebra morphism.

Finally, it is easy to see that $(\f r \com \f s)(f,h) = (f,h)$ for all $(f,h) \in \LGen(\omega,\Omega)$.
\end{proof}

%--------------------------------------------------------------------%
\section{Examples }
%\setcounter{equation}{0}
%--------------------------------------------------------------------%
In this section we recall results obtained for structures of the classical phase space. These results were the motivation of the paper.

We assume 
\emph{classical space-time} 
to be an oriented and time oriented
4--dimensional manifold
$\f E$
equipped with 
 a 
 Lorentzian metric
$g\,,$
with signature
$(1,3),$ \cite{JanMod08}. We denote by 
$ (x^{\lambda}) = (x^{0},x^{i}) $, $\lambda = 0,1,2,3,$ local coordinates on $ \f E $ such that $\der_0$ is time-like and $\der_i$ are space-like.
A 
\emph{motion} 
is defined to be a 1--dimensional time-like
submanifold of space-time.
We define the 
\emph{classical (Einsteinian) phase space} 
to be the open
subspace
${\M J}_1{\f E} \subset J_1({\f E},1)$
consisting of all 
1--jets (1st order contact elements) of motions.
So elements of ${\M J}_{1x}{\f E}$ are classes
of non-parametrized curves which have in a point $x\in \f E$ the same
tangent line lying inside the light cone. 
$\pi^1_0:{\M J}_1{\f E} \to {\f E}$ is a fibred manifold but not an
affine bundle!
We have the induced coordinate chart $(x^\lam, x^i_0)$.

The metric $g$ gives naturally the unscaled horizontal
\emph{time form} 
$$
\wha\tau
 :
{\M J}_1{\f E} \to T^*{\f E}\,, \quad \wha\tau = \wha\tau_\lambda \, dx^\lambda \,.
$$ 

%--------------------------------------------------------------------%
\subsection{Infinitesimal symmetries of the gravitational contact structure }
%\setcounter{equation}{0}
%--------------------------------------------------------------------%
The pair $(-\widehat\tau, \Omega\Elg)$, where 
$$
\Omega\Elg = - d\widehat\tau : 
{\M J}_1\f E \to \bigwedge^2T^*{\M J}_1\f E\,,
$$ 
is the 
contact (gravitational) regular
structure on $ {\M J}_1\f E $.
The dual Jacobi structure is given by a pair $ (-\wha\gamma\Elg,\Lambda\Elg) $, where $ \wha\gamma\Elg $ and  $\Lambda\Elg $ are naturally given by the metric field (for details see \cite{JanMod08}).

\medskip
By Corollary \ref{Cr: 2.8} infinitesimal symmetries of the gravitational contact phase structure
  are Hamilton-Jacobi lifts of conserved functions, i.e. they are of the type
$
X = df\Sha + f\,\wha\gamma\Elg\,,
$ 
where $\wha\gamma\Elg\tecka f = 0$ and, moreover, $f = \wha\tau(X) = \wha\tau(\underline{X}) $. Here
$\underline{X} = T\pi^1_0(X): {\M{J}}_1\f E \to T\f E$ is a \emph{generalized vector field} in the terminology of 
\cite{Olv86}.
So infinitesimal symmetries are of the type
\begin{equation}\label{Eq: 4.1}
X = d(\wha\tau(\underline{X}))\Sha + \wha\tau(\underline{X})\,\wha\gamma\Elg\,,
\end{equation}
where the following conditions are satisfied

\medskip
1. 
{(Projectability)}
 The Hamilton-Jacobi lift \eqref{Eq: 4.1}  projects on  $\underline{X}$.

\smallskip
2. {(Conservation)}
$\wha\tau(\underline{X})$ is conserved, i.e. $\gamma\Elg {\tecka} (\wha\tau(\underline{X})) = 0$.

\smallskip
The following results were proved in \cite{Jan14}.

\begin{lem}\label{Lm:4.1}
A symmetric $k$-vector field $\overset{k}{K}$, $ k\ge 1 $, on $\f E$ admits the generalized vector field satisfying the projectability condition. Such generalized vector fields are given by
\begin{equation*}
\underline{X}[\overset{k}{K}] = k\, \underset{(k-1)-times}{\underbrace{\wha\tau\con \dots\con \wha\tau \con}}\overset{k}{K} -  {(k-1)} \overset{k}{K}(\wha\tau,\dots,\wha\tau)\, \wha{\text{\K{d}}}: {\M J}_1\f E\to T\f E\,,
\end{equation*}
where $ \wha{\text{\K{d}}} = T\pi^{1}_0(\wha\gamma\Elg) $.
Then  we obtain the induced phase function
$$
\wha\tau\big(\underline{X}[\overset{k}{K}]\big)=\overset{k}{K}(\wha\tau) = \overset{k}{K}(\wha\tau,\dots,\wha\tau)
= \overset{k}{K}{}^{\lambda_1\dots\lambda_k}\,
\wha\tau_{\lam_1}\dots \wha\tau_{\lam_k}\,.\eqno\END
$$
\end{lem}

\begin{lem}\label{Lm:4.2}
 Let $\overset{0}{K}$ be a space-time function. Then
 $ \wha\gamma\Elg.\overset{0}{K} = 0 $ if and only if $\overset{0}{K}$ is a constant.
The phase function $ \overset{k}{K}(\wha\tau) $, $ k\ge 1 $, is conserved with respect to the gravitational Reeb vector field, i.e. $ \wha\gamma\Elg\tecka \overset{k}{K}(\wha\tau) = 0 $, if and only if $\overset{k}{K}$ is a Killing $k$-vector field.
\hfill\END
\end{lem}

\begin{thm}\label{Th: 4.3}
The Hamilton-Jacobi lift of a phase function 
\begin{equation}\label{Eq: 4.2}
K= \overset{0}{K} + \sum_{k\ge 1}\overset{k}{K}(\wha\tau) \,
\end{equation}
is an infinitesimal symmetry of the gravitational contact phase structure $ (-\wha\tau,\Omega\Elg) $ if and only if $\overset{0}{K}$ is a constant and $\overset{k}{K}$, $k\ge 1$,  are Killing $k$-vector fields. 
Moreover, for $k=1$ the corresponding infinitesimal symmetry coincides with the jet flow lift $\M J_1\overset{1}{K}$ and is projectable on space-time. 
For $k\ge 2$ the corresponding infinitesimal symmetry is hidden.
\hfill\END
\end{thm}

\begin{rmk}\label{Rm: 4.1}
{\rm 
It is very well known that Killing multivector fields generates 
on $T^*\f E$ \emph{functions  
constant of motion} 
(functions constant on lifts of geodesic curves) (see, for instance, \cite{Som73}).
In \cite{Jan15} it was proved that if we consider the mapping 
$
-\wha \tau: {\M J}_1\f E \to T^*\f E\,,
$
then a conserved phase function of type \eqref{Eq: 4.2} is obtained as the pull-back
$K=-\wha\tau^*(\tilde K)$
of a function $\tilde K$ constant of motion.
}\hfill\END
\end{rmk}

%--------------------------------------------------------------------%
\subsection{Infinitesimal symmetries of the total  almost-cosymplectic-contact phase structure}
%\setcounter{equation}{0}
%--------------------------------------------------------------------%
Let us assume an
\emph{electromagnetic (Max\-well) field} 
which is
 a  closed 2--form
$\wha F : \f E \to  \wedge^2T^*\f E \,.$
Then we can consider
the
total phase 2--form
$$
\Ome\Joi =: \Omega\Elg + \Ome\Ele = -d\wha\tau + \tfrac{1}{2}\,\widehat F \,
$$
and the pair
$(-\widehat \tau, \, \Ome\Joi)$
turns out to be an 
almost-cosymplectic-contact
structure of the phase space,
 i.e.
$\Ome\Joi$ 
is closed and 
$\widehat\tau\wed \Ome\Joi \wed \Ome\Joi \wed \Ome\Joi$
is a volume form. 

The 
dual almost-coPoisson-Jacobi pair 
is then given by the 
total Reeb vector field
$\wha\gam\Joi = \wha\gamma\Elg+ \wha\gamma\Ele$ and the
total  phase 2-vector
\color{black}
$\Lam\Joi= \Lambda\Elg +\Lambda\Ele$, where
$
\wha\gam\Ele
$
and $\Lambda\Ele$ are $\f E$-vertical given by $ g $ and $\wha F$, see \cite{JanMod08}. 

In \cite{Jan15} it was proved that
all phase infinitesimal symmetries of the total phase structure are vector fields of the type
\begin{equation*}%\label{Eq: 4.3}
X = d(\wha\tau(\underline{X}) + \breve{f})\Sha{}\Joi + \wha\tau(\underline{X})\, \wha\gamma\Joi\,
\end{equation*}
where  $\underline{X}$ is a generalized vector field and  $\breve{f}\in C^\infty(\f E)$ such that:

1) $ d\breve{f} = \underline{X}\con \wha F\,; $

2)
({Projectability})
The vector field $X$ projects on $ \underline{X} $;

3)
({Conservation})
$ \wha\gamma\Joi.(\wha\tau(\underline{X}) + \breve{f}) = 0 $.

\smallskip
The projectability condition is the same as in the case of the contact gravitational structure which follows from the fact that the fields $\wha\gamma\Ele$ and $\Lambda\Ele$ are $\f E$-vertical. So it is sufficient to describe conditions that the function \eqref{Eq: 4.2}, where $\breve{f} = \overset{0}{K}$, is conserved. 

\begin{thm}\label{Th: 4.4} 
{\rm \cite{Jan15}}
A phase function \eqref{Eq: 4.2}
is conserved, i.e. $ \wha\gam\Joi\tecka K = 0 $, if and only if
\begin{gather}\label{Eq: 4.4}
g^{\rho\lambda}\, \der_\rho\overset{0}{K} + \overset{1}{K}{}^{\rho}\, \wha F^{\lambda}{}_\rho
 =
0\,,
\\
\nabla^{(\lambda_1}\overset{k}{K}{}^{\lambda_2\dots\lambda_{k+1})} + 
(k+1)\, \overset{k+1}{K}{}^{\rho(\lambda_1\dots\lambda_k}\, \wha F^{\lambda_{k+1})}{}_\rho
 = \label{Eq: 4.5}
0\,,
\end{gather}
for $ k \ge  1 $\,.
\hfill$\square$
\end{thm}

\begin{crl}\label{Cr: 4.5}
Let us assume a special phase function
$
K = \overset{0}{K} + \overset{1}{K}(\wha\tau)\,.
$
Then the conditions \eqref{Eq: 4.4} and \eqref{Eq: 4.5} are reduced to
$$
\der_\rho \overset{0}{K} - \overset{1}{K}{}^{\sigma}\, \wha F_{\sigma\rho} = 0\,,\qquad \nabla^{(\lambda_1}\overset{1}{K}{}^{\lambda_2)} = 0
$$
and we obtain the result of 
\cite{JanVit12}, i.e. $\overset{1}{K} $ is a Killing vector field and $\overset{0}{K} $ and  $\overset{1}{K} $ are related by the formula $d\overset{0}{K} = \overset{1}{K}\con \wha F$. Moreover, the corresponding infinitesimal symmetry
is the jet flow lift ${\M J}_1\overset{1}{K} $ which
projects on $\overset{1}{K}$.
\hfill$\square$
\end{crl}

\begin{rmk}\label{Rm: 4.2}
{\rm
Let us assume a phase function
$
K =  \overset{k}{K}(\wha\tau)\,, \quad k\ge 2\,.
$
Then the condition \eqref{Eq: 4.5} gives
$$
\nabla^{(\lambda_1}\overset{k}{K}{}^{\lambda_2\dots\lambda_{k+1})} = 0\,,\qquad
\overset{k}{K}{}^{\rho(\lambda_1\dots\lambda_{k-1}}\, \wha F^{\lambda_k)}{}_\rho = 0\,
$$
and we obtain that $ \overset{k}{K} $ is a \emph{Killing--Maxwell} $k$-vector field. 
But the corresponding lift has to satisfy also the condition 1)  
which is of the form $\underline{X}[K]\con \wha F = 0 $. 
This condition implies $\wha{\text{\K{d}}} \,\con\, \wha F = 0\,,$ 
which implies $ \wha F \equiv 0$ and the structure 
is reduced to the gravitational case.
This implies that there are no non-projectable (hidden) infinitesimal symmetries generated by Killing--Maxwell $k$-vector fields for $k\ge 2$.

So all infinitesimal symmetries of $(-\wha\tau,\Ome\Joi)$ are projectable and can be generated by pairs $(\underline{X},\breve{f})$ of Killing vector fields and spacetime functions such that 
$
d\breve{f} = \underline{X}\con \wha F.
$
Such pairs are sections of the Lie algebroid $T\f E\oplus \B R \to \f E$ with the bracket $\db[ ; \db]_{\wha F}$ (see Section \ref{Sec: LA} and
\cite{JanVit12}).
The sections of the Lie algebroid described in Section \ref{Sec: LA} are obtained as the 1-jet flow lifts of $ \underline{X} $ and the pull-backs of $ \breve{f} $.
}\hfill\END
\end{rmk}

%--------------------------------------------------------------------%

%--------------------------------------------------------------------%
\end{document}